\documentclass[a4paper,11pt]{amsart}
\usepackage[cyr]{aeguill}
\usepackage[pdfdisplaydoctitle=true,
            colorlinks=true,
            urlcolor=blue,
            citecolor=blue,
            linkcolor=blue,
            pdfstartview=FitH,
            pdfpagemode=None,
            bookmarksnumbered=true]{hyperref}
\newtheorem{thm}{Theorem}[section]
\newtheorem{cor}[thm]{Corollary}

\newtheorem{prop}[thm]{Proposition}
\theoremstyle{definition}

\theoremstyle{remark}
\newtheorem{rem}[thm]{Remark}
\numberwithin{equation}{section}
\newcommand{\norm}[1]{\left\Vert#1\right\Vert}

\newcommand{\RR}{\mathbb{R}}
\newcommand{\CC}{\mathbb{C}}

\newcommand{\PP}{\mathbb{P}}

\newcommand{\dbar}{\bar{\partial}}
\newcommand{\zbar}{\bar{z}}

\newcommand{\pd}[2]{\frac{\partial #1}{\partial #2}}
\newcommand{\scal}[2]{\left\langle #1 , #2 \right\rangle}

\DeclareMathOperator{\Ric}{Ric}
\DeclareMathOperator{\grad}{grad}
\DeclareMathOperator{\Hess}{Hess}
\hypersetup{
    pdftitle={A Lichnerowicz estimate for the first eigenvalue of convex domains in K\"{a}hler manifolds}, %  Title
    pdfauthor={Vincent Guedj, Boris Kolev, Nader Yeganefar}, % Author
    pdfsubject={MSC 2010: 58C40, 53C21}, % Subject
    pdfkeywords={Lichnerowicz estimate, first eigenvalue, convex domains in K\"{a}hler manifold}, % Keywords
    pdflang=EN % Language
    }

\begin{document}

\title[A Lichnerowicz estimate]{A Lichnerowicz estimate for the first eigenvalue of convex domains in K\"{a}hler manifolds}%

\author{Vincent Guedj}
\address{Institut Universitaire de France et Institut de Math\'{e}matiques de Toulouse, Universit\'{e} Paul Sabatier,
31062 Toulouse cedex 09, France}
\email{vincent.guedj@math.univ-toulouse.fr}

\author{Boris Kolev}
\address{LATP, CNRS \& Universit\'{e} de Provence, 39 Rue F. Joliot-Curie, 13453 Marseille Cedex 13, France}
\email{kolev@cmi.univ-mrs.fr}

%    Information for third author
\author{Nader Yeganefar}
\address{LATP, Universit\'{e} de Provence, 39 Rue F. Joliot-Curie, 13453 Marseille Cedex 13, France}
\email{Nader.Yeganefar@cmi.univ-mrs.fr}

%\thanks{}%


%\date{}%
%\dedicatory{}%
%\commby{}%
% ----------------------------------------------------------------
\begin{abstract}
In this article, we prove a Lichnerowicz estimate for a compact \emph{convex} domain of a K\"{a}hler manifold whose Ricci curvature satisfies $\Ric \geq k$ for some constant $k>0$. When equality is achieved, the boundary of the domain is totally geodesic and there exists a nontrivial holomorphic vector field.

We show that a ball of sufficiently large radius in complex projective space provides an example of a \emph{strongly pseudoconvex} domain which is \emph{not convex}, and for which the \emph{Lichnerowicz estimate} fails.
\end{abstract}

\maketitle
% ----------------------------------------------------------------

\section{Introduction}
\label{sec:intro}

Let $(M^{n},g)$ be a compact $n$-dimensional Riemannian manifold. Assume first that $M$ has no boundary. A theorem of Lichnerowicz \cite{Lic58} asserts that if the Ricci curvature $\Ric$ of $M$ satisfies $\Ric \geq k$ for some constant $k>0$, then the first nonzero eigenvalue $\lambda$ of the Laplace operator satisfies
\begin{equation}\label{equ:lichnerowicz}
    \lambda \geq \frac{n}{n-1}k.
\end{equation}

Here, $nk/(n-1)$ should be viewed as the first nonzero eigenvalue of the round $n$-dimensional sphere $S^{n}(k/(n-1))$ of constant curvature $k/(n-1)$. Moreover, by a result of Obata \cite{Oba62}, the equality case in \eqref{equ:lichnerowicz} is obtained if and only if $M$ is isometric to this sphere. Reilly considered a similar problem, but for compact manifolds with boundary \cite{Rei77}. Namely, he proved that if $M$ is as in Lichnerowicz theorem, except that it has now a boundary such that its mean curvature with respect to the outward normal vector field is nonnegative, then the first eigenvalue $\lambda$ of the Laplace operator with the Dirichlet boundary condition still satisfies \eqref{equ:lichnerowicz}. He also proved that the equality case characterizes a hemisphere in $S^{n}(k/(n-1))$.

\smallskip

In another direction, Lichnerowicz showed that for K\"{a}hler manifolds, his estimate \eqref{equ:lichnerowicz} can be improved. To explain this, we modify slightly our normalization conventions: we consider a closed K\"{a}hler manifold $M$ of \emph{real} dimension $n=2m$, whose Ricci curvature satisfies $\Ric \geq k>0$. On a K\"{a}hler manifold, there are \textit{a priori} three Laplace operators, namely the usual Laplace operator associated to exterior differentiation, and two Laplace operators associated to $\partial$ and $\dbar$ respectively. Moreover, the latter two are actually equal and coincide with half the usual Laplacian. In the sequel, we will consider the Laplace operator associated with $\dbar$, and we will denote it by $\Delta$, so that
\begin{equation*}
    \Delta = \dbar^{*}\dbar + \dbar \dbar^{*}.
\end{equation*}
The first nonzero eigenvalue of $\Delta$ acting on functions will be denoted by $\lambda$. In view of these conventions, the condition $\Ric \geq k$ and the Lichnerowicz estimate \eqref{equ:lichnerowicz} give
\begin{equation*}
    \lambda \geq \frac{mk}{(2m-1)}.
\end{equation*}
Lichnerowicz improved this bound by showing that
\begin{equation*}
    \lambda \geq k
\end{equation*}
Moreover, if equality is achieved, then there is a non trivial holomorphic vector field on $M$.

\smallskip

The purpose of this note is to consider the case of compact K\"{a}hler manifolds with boundary. As in Reilly's result, we will have to impose some convexity property on the boundary:

\begin{thm}\label{thm:estimee}
Let $M$ be a compact convex domain in a K\"{a}hler manifold. Assume that the Ricci curvature satisfies $\Ric \geq k$ for some constant $k>0$. Then the first nonzero eigenvalue $\lambda$ of the Laplacian with Dirichlet boundary condition satisfies
\begin{equation*}
    \lambda \geq k.
\end{equation*}
Moreover, if equality is achieved, then the boundary $\partial M$ is totally geodesic and there is a nontrivial holomorphic vector field on $M$.
\end{thm}

\begin{rem}
As we will see in the proof, the convexity hypothesis may be relaxed into another condition of mean curvature type (see condition \eqref{equ:assez} below). However, this condition has no clear geometrical meaning, so that we have stated our theorem with the convexity hypothesis instead.
\end{rem}

\begin{rem}
It is natural to ask whether our result remains true if one assumes the pseudoconvexity of the boundary instead of its convexity. It turns out that a ball of sufficiently large radius in complex projective space provides an example of a strongly pseudoconvex domain which is not convex, and for which the Lichnerowicz estimate fails (see Proposition \ref{prop:cex} for more details on this).
\end{rem}

\begin{rem}
In the real setting, one can consider the Laplacian with the Neumann boundary condition, and again with the convexity condition, one can show that the Lichnerowicz estimate \eqref{equ:lichnerowicz} still holds for the first nonzero eigenvalue \cite{PMYC86}. In the K\"{a}hler setting, by using the method of proof of Theorem \ref{thm:estimee}, it should also be possible to prove that the conclusion of this theorem is true for the first nonzero eigenvalue of the $\dbar$-Laplacian with the absolute $\dbar$-condition on the boundary.
%We leave it to the interested reader to write down an actual proof of this claim.
\end{rem}

An immediate consequence of our theorem is

\begin{cor}
Assume that $M$ is a strongly convex domain in a complex manifold which can be endowed with a K\"{a}hler metric whose Ricci curvature satisfies $\Ric \geq k$ for some constant $k>0$. Then the first nonzero eigenvalue $\lambda$ of the Laplacian with Dirichlet boundary conditions satisfies
\begin{equation*}
    \lambda > k.
\end{equation*}
\end{cor}

Our proof follows the same strategy as in the original proofs of Lichnerowicz and Reilly. We use an appropriate Bochner formula for the Laplacian acting on $(0,1)$-forms and apply it to $\overline{\partial} f$, where the function $f$ is an eigenfunction of $\Delta$ for the first eigenvalue. After integrating the result on $M$ and integrating by parts, the desired eigenvalue estimate follows if we can prove that some boundary term is nonpositive, which is the case under the convexity hypothesis.

% ----------------------------------------------------------------

\section{Background material}
\label{sec:background}

In this section, we recall some well-known facts that will be used in the proof of our main result.

\subsection{Decomposition of the Hessian}

Let $f$ be a real valued smooth function on a K\"{a}hler manifold $(M,J,g)$. Its Riemannian Hessian $\Hess f$ is as usual defined by
\begin{equation*}
    \Hess f = \nabla df
\end{equation*}
This Hessian may be decomposed as the sum of a $J$-symmetric bilinear form and a $J$-skew-symmetric bilinear form. More specifically, we have
\begin{equation*}
    \Hess f = H^{1} f + H^{2} f
\end{equation*}
where for tangent vectors $A$ and $B$,
\begin{equation*}
    H^{1} f (A,B) = \frac{1}{2}\left\{\Hess f (A,B) + \Hess f (JA,JB)\right\}
\end{equation*}
and
\begin{equation*}
    H^{2} f (A,B) = \frac{1}{2}\left\{\Hess f (A,B) - \Hess f (JA,JB)\right\}
\end{equation*}

The two following facts may easily be checked:
\begin{enumerate}
 \item The $(1,1)$-form associated to $H^{1}f$ by the complex structure $J$ is $i\partial\bar{\partial}f$:
\begin{equation*}
    H^{1} f (JA,B) = i\partial\dbar f(A,B).
\end{equation*}
  \item In local coordinates, $H^{2}f$ has the following components
\begin{equation*}
    (H^{2}f)_{pq} = \overline{(H^{2}f)_{\bar{p}\bar{q}}} =
    \frac{\partial^{2}f}{\partial z_{p}\partial z_{q}} - \Gamma_{pq}^{r}\pd{f}{z_{r}},
\end{equation*}
and the other components vanish. $H^{2}f$ is called the \emph{complex Hessian}.
\end{enumerate}

\subsection{Bochner formula for the Laplacian}

Let $(M,g)$ be a K\"{a}hler manifold, and denote by $D$ its Levi-Civita connection. If $\alpha$ is a $(0,1)$-form, we denote by $D''\alpha$ the $(0,2)$-part of $D\alpha$. More precisely, $D\alpha$ is a section of the bundle $T^*M\otimes (T^*)^{0,1}M$; this bundle decomposes as a direct sum
\begin{equation*}
    ((T^*)^{1,0}M\otimes (T^*)^{0,1}M)\oplus ((T^*)^{0,1}M\otimes (T^*)^{0,1}M),
\end{equation*}
and $D''\alpha$ is the projection of $D\alpha$ on the second factor of this decomposition. In local complex coordinates, we have
\begin{equation*}
    (D''\alpha )_{\bar{p}\bar{q}} = \pd{\alpha_{\bar{q}}}{\zbar_{p}} - \Gamma_{\bar{p}\bar{q}}^{\bar{r}}\alpha_{\bar{r}}.
\end{equation*}
Let now $(D'')^*$ be the formal adjoint of $D''$. For a section $\beta$ of $(T^*)^{0,1}M\otimes (T^*)^{0,1}M$ one can see that locally
\begin{equation*}
    ((D'')^{*}\beta)_{\bar{p}} = - g^{q\bar{r}} \pd{\beta_{\bar{r}\bar{p}}}{z_{q}}.
\end{equation*}
Then we have the following Bochner formula for the $\dbar$-Laplacian $\Delta$ acting on $(0,1)$-forms:
\begin{equation}\label{equ:Bochner}
    \Delta = (D'')^*D''+\Ric .
\end{equation}
For future reference, we also give the integration by parts formula for $D''$ in the presence of a boundary (see e.g. \cite[Proposition 9.1]{Tay96}). Here, we assume that $M$ is compact, and we let $n$ denote the outward unit normal vector field on $\partial M$. The $(0,1)$ part of the dual $1$-form $\nu$ corresponding to $n$ by the metric will be denoted by $\nu ^{0,1}$. Finally, we let $\sigma$ denote the measure induced on the boundary by the metric. For smooth $\alpha$ and $\beta$, we then have
\begin{equation}\label{equ:IPP}
    \scal{D''\alpha}{\beta}_{L^{2}(M)} = \scal{\alpha}{(D'')^*\beta}_{L^{2}(M)} + \int _{\partial M} \scal{\nu ^{0,1}\otimes \alpha}{\beta} \, \sigma .
\end{equation}

% ----------------------------------------------------------------

\section{Estimate of the first eigenvalue}
\label{sec:estimate}

In this section, $M$ is a compact smooth domain in a K\"{a}hler manifold of complex dimension $m$, with metric $g$ and Ricci curvature bounded from below by some positive constant $k$. The outward unit normal vector field  on the boundary $\partial M$ is denoted by $n$ and its covariant associated $1$-form by $\nu$.

\subsection{Bochner formula and the first eigenvalue}
\label{subsec:Bochner_formula}

Let $f$ be a real valued eigenfunction of $\Delta$ corresponding to the first nonzero eigenvalue $\lambda$ of $\Delta$, so that $f: \overline{M}\to \RR$ is smooth, vanishes on the boundary $\partial M$, and satisfies $\Delta f=\lambda f$. (Note that it is possible to choose $f$ to be real valued, because $\Delta$ is equal to half the usual Laplacian.) We write the Bochner formula \eqref{equ:Bochner} for the $(0,1)$-form $\dbar f$ and take the $L^{2}$-inner product of the resulting equality with $\dbar f$ itself:
\begin{equation}\label{equ:Bochner2}
    \scal{\Delta \dbar f}{\dbar f} _{L^{2}(M)} = \scal{(D'')^*D''\dbar f}{\dbar f}_{L^{2}(M)} + \int _M \Ric{(\dbar f, \dbar f)} .
\end{equation}
Using the fact that $\Delta \dbar = \dbar \Delta$ and that $f\vert_{\partial M}=0$, we can integrate by parts the left hand side of \eqref{equ:Bochner2} to get
\begin{equation*}
\begin{split}
  \scal{\Delta \dbar f}{\dbar f} _{L^{2}(M)} & = \scal{\dbar \Delta f}{\dbar f}_{L^{2}(M)} \\
    & =  \scal{\dbar (\lambda f)}{\dbar f}_{L^{2}(M)} \\
    & = \lambda \scal{\Delta f}{f}_{L^{2}(M)} \\
    & = \lambda^{2} \norm{f}^{2}_{L^{2}(M)} .
\end{split}
\end{equation*}
We can deal with the Ricci term in the right hand side of \eqref{equ:Bochner2} in a similar way
\begin{equation*}
\begin{split}
  \int_M \Ric{(\dbar f, \dbar f)} & \ge k \scal{\dbar f}{\dbar f}_{L^{2}(M)} \\
    & = k \scal{\Delta f}{f}_{L^{2}(M)} \\
    & = k \lambda \norm{f}^{2}_{L^{2}(M)}.
\end{split}
\end{equation*}
Finally, we can integrate by parts the first term in the right hand side of \eqref{equ:Bochner2} (see formula \eqref{equ:IPP}):
\begin{equation}\label{equ:IPP1}
    \scal{(D'')^*D''\dbar f}{\dbar f}_{L^{2}(M)} = \norm{D'' \dbar f}^{2}_{L^{2}(M)} - \int_{\partial M} \scal{D''\dbar f}{\nu ^{0,1} \otimes \dbar f}\, \sigma ,
\end{equation}
and combining this with our previous estimates, we obtain
\begin{equation}\label{equ:egalite}
    \lambda (\lambda -k)\norm{f}^{2}_{L^{2}(M)} \geq \norm{D'' \dbar f}^{2}_{L^{2}(M)} - \int _{\partial M} \scal{D''\dbar f}{\nu ^{0,1} \otimes \dbar f}\, \sigma .
\end{equation}
As a consequence, if we set
\begin{equation*}
    I = -\int _{\partial M} \langle D''\dbar f , \nu ^{0,1} \otimes \dbar f\rangle\, \sigma ,
\end{equation*}
we will get  $\lambda \geq k$ provided we can prove that $I\geq 0$. In the next subsection, we will see that this is indeed the case under suitable assumptions on the boundary.

\subsection{Boundary term}
\label{subsec:boundary_term}

To estimate the boundary term $I$, we first notice that as $f$ is real valued, we have
\begin{equation*}
    (D''\dbar f)_{\bar{p}\bar{q}} = (H^{2}f)_{\bar{p}\bar{q}}
\end{equation*}
so that
\begin{equation*}
    I = -\int_{\partial M} \scal{H^{2} f}{\nu^{0,1} \otimes \dbar f} \, \sigma
    = -\int_{\partial M} H^{2} f (n^{0,1} , (\partial f)^{\sharp} ) \, \sigma  .
\end{equation*}
We then choose a boundary defining function $\rho$ for $\partial M$. This means that $\rho$ is a smooth real valued function such that $M=\{ \rho \leq 0\}$, $\partial M=\{\rho =0\}$ and $d\rho$ does not vanish on $\partial M$. By multiplying $\rho$ by a suitable smooth positive function if necessary, we may assume that
\begin{equation*}
    n = \grad \rho .
\end{equation*}
Moreover, near a fixed (but arbitrary) point of the boundary $\partial M$, we fix a local orthonormal frame adapted to the complex structure $J$ which has the form
\begin{equation*}
    v_{1}, Jv_{1}, \dotsc , v_{m}, Jv_{m}=n=\grad\rho .
\end{equation*}
We also set
\begin{equation*}
    e_{p} = \frac{1}{\sqrt{2}}(v_{p} - iJv_{p}), \quad p=1, \dotsc , m.
\end{equation*}
Note that as $f$ vanishes on $\partial M$, its derivatives along tangent vectors to $\partial M$ also vanish and consequently
\begin{equation*}
    (\partial f)^{\sharp} = \frac{-i}{\sqrt{2}} (n\cdot f) \bar{e}_{m}, \qquad n^{0,1} = \frac{-i}{\sqrt{2}}\bar{e}_{m},
\end{equation*}
where $n \cdot f$ means $df(n)$. Therefore,
\begin{equation*}
    I = \frac{1}{2} \int_{\partial M} (n \cdot f) \Hess f (\bar{e}_{m}, \bar{e}_{m}) \, \sigma
\end{equation*}
which can be decomposed  as $I = I_{1} + i I_{2}$ with
\begin{equation*}
    I_{1} = \frac{1}{4}\int_{\partial M} (n \cdot f)\left[ \Hess f (Jn,Jn) - \Hess f (n,n) \right]\, \sigma
\end{equation*}
and
\begin{equation*}
    I_{2} = \frac{-1}{2}\int_{\partial M} (n \cdot f)\Hess f (Jn,n)\, \sigma  .
\end{equation*}
Actually $I_2$ vanishes because $I$ is a real number. (This follows from the fact that in equation  \eqref{equ:Bochner2}, the left hand side and the Ricci term are real numbers, so that the term involving $D''$ is also a real number. This implies, by equation \eqref{equ:IPP1}, that the boundary term $I$ is a real number as well. There is also a more conceptual reason for the vanishing of $I_2$, see section~\ref{subsec:direct_proof}).
We now turn our attention to $I_1$. As $\triangle f = \lambda f = 0$ on $\partial M$, the trace of $\Hess f$ is also zero on $\partial M$:
\begin{multline*}
    \Hess f (Jn,Jn) - \Hess f (n,n) = \sum_{k=1}^{m-1} \left[ \Hess f (v_{k},v_{k}) + \Hess f (Jv_{k},Jv_{k})\right] \\
    + 2\Hess f (Jn,Jn).
\end{multline*}
We notice that all vectors appearing in the right hand side are tangent to the boundary. For such a vector $u$, we have on $\partial M$
\begin{equation*}
\begin{split}
  \Hess f (u,u) & = - \scal{\nabla_{u}u}{n} (n \cdot f) \\
    & = \scal{\nabla_{u}n}{u} (n \cdot f) \\
    & = (n \cdot f) \Hess \rho (u,u).
\end{split}
\end{equation*}
This implies
\begin{multline}\label{equ:bord}
    I_{1} = \frac{1}{4}\int_{\partial M} (n \cdot f)^{2}\Big\{ \sum_{k=1}^{m-1} [\Hess \rho (v_{k},v_{k})+\Hess\rho (Jv_k, Jv_k)] \\
    + 2\Hess \rho (Jn,Jn) \Big\} \, \sigma  .
\end{multline}
If we assume that $\partial M$ is convex, then all terms in the integrand of the right hand side are nonnegative, so that $I=I_1\geq 0$ as desired. This proves that $\lambda \geq k$ in the convex case.

It remains to deal with the equality case. If we assume that $\lambda =k$, then by equation \eqref{equ:egalite}, we must have $D''\dbar f =0$ and $I=0$. On the one hand, $D''\dbar f=0$ means that the $(1,0)$-vector field associated to $\dbar f$ by the metric is a (nonzero) holomorphic vector field. On the other hand, from $I=0$, we infer that the integrand in equation~\eqref{equ:bord} has to vanish identically on the boundary:
\begin{equation*}
    (n \cdot f)^{2}\Big\{ \sum_{k=1}^{m-1} \{ \Hess \rho (v_{k},v_{k}) + \Hess \rho (Jv_k, Jv_k)\} + 2\Hess \rho (Jn,Jn) \Big\}=0.
\end{equation*}
Assume by contradiction that $\partial M$ is not totally geodesic (but still convex of course). Then the term between the brackets is positive at some point and we will get the vanishing of $n.f$ on an open subset of $\partial M$. But $f$ is in the kernel of the elliptic operator $\Delta -\lambda$ and vanishes on $\partial M$. By the unique continuation principle for elliptic operators (see e.g. \cite{BW93}), $f$ has to vanish on $M$ as well, which is absurd. Therefore, $\partial M$ is totally geodesic. This completes the proof of Theorem \ref{thm:estimee}.\\

\subsection{A direct proof that the boundary term is real}
\label{subsec:direct_proof}

The fact that
\begin{equation*}
    I_{2} = \frac{-1}{2}\int_{\partial M} (n \cdot f)\Hess f (Jn,n)\, \sigma
\end{equation*}
vanishes is also a consequence of the fact that the expression
\begin{equation*}
    (n \cdot f)\Hess f (Jn,n)\, \sigma = (n \cdot f)(Jn \cdot n \cdot f)\, \sigma
\end{equation*}
is an exact differential form on the closed manifold $\partial M$.  Indeed, the vector field $Jn = J \grad \rho$ is the Hamiltonian vector field associated to $\rho$. This means that if $\omega$ is the K\"{a}hler form, then
\begin{equation*}
    i_{Jn} \, \omega = - d\rho.
\end{equation*}
Hence
\begin{equation*}
    d \, i_{Jn} \,i_{n} \, \omega^{m} = - md(n \cdot \rho)\wedge \omega^{m-1} - m(m-1)d\rho \wedge d\,i_{n} \, \omega \wedge \omega^{m-2}
\end{equation*}

Let $j : \partial M \to M$ be the inclusion map. Since the functions $n \cdot \rho $ and $\rho$ are constant on $\partial M$, we have
\begin{equation*}
    j^{*}\left(d \, i_{Jn} \,i_{n} \, \omega^{m} \right) = 0.
\end{equation*}
Now, $Jn$ is a vector field defined on a neighborhood of  $\partial M$ whose restriction to $\partial M$ is tangent to $\partial M$, so that
\begin{equation*}
    j^{*}(i_{Jn} \,\beta) = i_{Jn}j^{*}(\beta)
\end{equation*}
for any differential form $\beta$. As a consequence, we get
\begin{equation*}
    d \, i_{Jn} j^{*}\left(i_{n} \, \omega^{m} \right) = 0.
\end{equation*}
Finally, we have
\begin{equation*}
    j^{*}\left(i_{n}\,\omega^{m}\right) = \sigma ,
\end{equation*}
and
\begin{equation*}
    d \, i_{Jn} \sigma = 0.
\end{equation*}
Defining a vector field $X$ by
\begin{equation*}
    X = \frac{1}{2}(n \cdot f)^{2} Jn,
\end{equation*}
it follows that on $\partial M$, we have
\begin{equation*}
    d\,i_{X} \, \sigma = (n \cdot f)(Jn \cdot n \cdot f)\sigma.
\end{equation*}

% ----------------------------------------------------------------

\section{Counter-example in the pseudoconvex case}
\label{sec:counter-example}

We use the notation introduced in the previous section. It is clear from the proof of Theorem \ref{thm:estimee} that in order to get the estimate $\lambda \geq k$, it is enough to assume that on the boundary we have
\begin{equation}\label{equ:assez}
    \sum_{k=1}^{m-1} \{ \Hess \rho (v_{k},v_{k}) + \Hess \rho (Jv_k, Jv_k)\} + 2\Hess \rho (Jn,Jn) \geq 0,
\end{equation}
and not necessarily the convexity of $\partial M$. We may rewrite this condition as
\begin{equation*}
    \sum_{k=1}^{m-1}  H^{1} \rho (v_{k},v_{k}) + \Hess \rho (Jn,Jn) \geq 0.
\end{equation*}
Here, $\sum_{k=1}^{m-1}  H^{1} \rho (v_{k},v_{k})$ is the trace of the Levi form of the boundary, which would be nonnegative if $\partial M$ were assumed to be only pseudoconvex. The extra term $\Hess \rho (Jn,Jn)$, however, can usually not be controlled in the pseudoconvex case. This suggests that the conclusion of Theorem~\ref{thm:estimee} does not generally hold in this case, as we now explain.

We consider here the complex $m$-dimensional projective space $\PP^m(\CC)$ equipped with the Fubini-Study metric normalized so that the holomorphic sectional curvature is 4 (the Einstein constant is thus $2(m+1)$ and the diameter is $\pi/2$).

\begin{prop}\label{prop:cex}
Fix some point $x\in \PP^m(\CC )$, some $r_0\in ]0,\pi /2[$, and let $M$ be the geodesic ball centered at $x$, of radius $r_0$.
\begin{itemize}
 \item[(i)] If $r_0\in ]\pi /4, \pi/2[$, then $M$ is strongly pseudoconvex, not convex.
 \item[(ii)] The first nonzero eigenvalue of $M$ with Dirichlet boundary conditions goes to $0$ as $r_0$ approaches $\pi /2$
\end{itemize}
\end{prop}

\begin{proof}
The first point is a well-known result. For completeness, we outline the proof here. Denote by $r$ the distance function from $x$, and set $\rho =r^{2}-r_0^{2}$, so that $\rho$ is a smooth defining function for $M$. We want to compute the eigenvalues of the Hessian of $\rho$. As
\begin{equation*}
    \Hess \rho = 2r \Hess r +2 dr\otimes dr,
\end{equation*}
we will only have to compute the eigenvalues of $\Hess r$. To do this, we proceed as in the proof of \cite[Theorem A, p.19]{GW79}. Recall that for a tangent vector $u$, the curvature $R(u,.)u$ of $\PP ^m(\CC )$ is given by (\cite[Proposition F.34]{BGM71})
\begin{equation*}
    R(u,.)u =
    \left\{
       \begin{array}{ll}
         0, & \hbox{on $\RR u$;} \\
         4\mathrm{Id}, & \hbox{on $\RR Ju$;} \\
         \mathrm{Id}, & \hbox{on the orthogonal complement of $(u, Ju)$.}
       \end{array}
     \right.
\end{equation*}

Let $\gamma$ be a normal geodesic starting from $x$. We can choose a parallel frame along $\gamma$ which has the form $v_1, Jv_1, \ldots, v_m, Jv_m =\grad r$. Using the explicit expression of $R$, it is then easy to check that the space of Jacobi fields $V$ along $\gamma$ satisfying $V(0)=0$ and $V\perp \dot{\gamma}$ has as basis $V _i= \sin{(r)}v_i$, $JV_i$, $i=1,\ldots , m-1$ and $V_m=\sin{(2r)}v_m$. Using the second variation formula, we see that $\Hess r$ is diagonalized in the basis $v_1, Jv_1, \ldots , v_m, Jv_m$ with eigenvalues $\cot{(r)}$ (of order $2m-2$), $2\cot(2r)$ and $0$. If $r=r_0\in ]\pi /4, \pi /2[$, we infer that the Levi form of $\rho$ is positive definite, being equal to $2r_0\cot{(r_0)}\rm{Id}$ on the Levi distribution. In other words, $M$ is strongly pseudoconvex. However, $M$ is not convex because the principal curvature $2\cot{(2r_0)}$ is negative.

As for the second point of our proposition, it is for example a consequence of a more general result due to Chavel and Feldman \cite[Theorem 1]{CF78} which states the following: let $X$ be a compact Riemannian manifold and let $X'\subset X$ be a submanifold. For small $\varepsilon >0$, let $X'_\varepsilon$ be the $\varepsilon$-neighborhood of $X'$ in $X$ and denote by $\Omega _\varepsilon$ the set $X\setminus X' _\varepsilon$. Let $(\lambda _j)$ be the spectrum of $X$ and let $(\lambda _j(\varepsilon ))$ be the spectrum of $\Omega _\varepsilon$ with Dirichlet boundary conditions. If the codimension of $X'$ in $X$ is at least $2$, then for all $j$, $\lambda _j(\varepsilon)\to \lambda _{j-1}$ as $\varepsilon \to 0$. In our case, we can take $X=\PP ^m(\CC )$ and $X'=\PP ^{m-1}(\CC )$ which we view as the cut locus of our fixed point $x$. If $\varepsilon = \pi /2-r_0$, then $\Omega _{\varepsilon}$ coincides actually with $M$ and we get (ii).
\end{proof}

% ----------------------------------------------------------------

\end{document}